\documentclass[letterpaper, 10 pt, conference]{ieeeconf}  % Comment this line out if you need a4paper
\usepackage[T1]{fontenc}
\usepackage{anyfontsize}
\usepackage{xcolor}

\IEEEoverridecommandlockouts                              % This command is only needed if 
   \usepackage{amsmath,amsthm,amssymb}

\newtheorem{thm}{Theorem}
\newtheorem{cor}{Corollary}
\newtheorem{lemma}{Lemma}

\theoremstyle{definition}
\newtheorem{defn}{Definition}
\newtheorem{rmk}{Remark}

\newcommand{\tA}{\tilde{A}}
\newcommand{\tB}{\tilde{B}}
\newcommand{\tC}{\tilde{C}}
\newcommand{\tD}{\tilde{D}}

\renewcommand{\Re}{\text{Re}}% \,}
\title{\LARGE \bf 
$H_\infty$ Minimal Destabilizing Feedback for \\ Vulnerability Analysis and Attack Design of Nonlinear Systems*
}

\author{Gavin Glenn$^{1}$ and Emma J. Reid$^{1}$% <-this % stops a space
\thanks{*Notice: Office of Science of the U.S. Department of Energy. This manuscript has been authored by UT-Battelle, LLC, under contract DE-AC05-00OR22725 with the US Department of Energy (DOE). The US government retains and the publisher, by accepting the article for publication, acknowledges that the US government retains a nonexclusive, paid-up, irrevocable, worldwide license to publish or reproduce the published form of this manuscript, or allow others to do so, for US government purposes. DOE will provide public access to these results of federally sponsored research in accordance with the DOE Public Access Plan (http://energy.gov/downloads/doe-public-access-plan). This work is sponsored by the Department of Homeland Security (DHS) Science and Technology Directorate. Any opinions contained herein are those of the author and do not necessarily reflect those of DHS S\&T, UT-Battelle, or DOE}
\thanks{$^{1}$Cyber Resilience and Intelligence Division, Oak Ridge National Laboratory,
        1 Bethel Valley Rd, Oak Ridge, TN 37830, USA
        {\tt\small glennga@ornl.gov}}%
}

\begin{document}

\maketitle
\thispagestyle{empty}
\pagestyle{empty}
%%%%%%%%%%%%%%%%%%%%%%%%%%%%%%%%%%%%%%%%%%%%%%%%%%%%%%%%%%%%%%%%%%%%%%%%%%%%%%%%
\begin{abstract}

The robust stability problem involves designing a controlled system which remains stable in the presence of modeling uncertainty. In this context, results known as small gain theorems are used to quantify the maximum amount of uncertainty for which stability is guaranteed. These notions inform the design of numerous control systems, including critical infrastructure components such as power grids, gas pipelines, and water systems. However, these same concepts can be used by an adversary to design a malicious feedback attack, of minimal size, to drive the closed-loop system to instability. In this paper, we first present a detailed review of the results in robust control which allow for the construction of minimal destabilizers. These minimally sized attacks merely push the system to the stability boundary, which we demonstrate {\em do not necessarily destabilize nonlinear systems} even when the linearization is destabilized. Our main result leverages linear perturbation theory to {\em explicitly} prove, in the state space context, that internal destabilization is guaranteed for a broad class of nonlinear systems when the gain of these attacks is slightly increased.

\end{abstract}

%%%%%%%%%%%%%%%%%%%%%%%%%%%%%%%%%%%%%%%%%%%%%%%%%%%%%%%%%%%%%%%%%%%%%%%%%%%%%%%%
\section{INTRODUCTION}

Robust control seeks to design a feedback controller which guarantees system stability in the presence of modeling uncertainty \cite{ZhouTextbook}. However, an adversary can also use the very same design principles to exploit a system. By leveraging or creating feedback within the system, they can achieve a variety of undesirable outcomes. The effects of feedback attacks are further exacerbated by the connectivity of modern systems, making it difficult to eliminate this attack vector. In \cite{Rai2012, Chetty, ChettyBook_2020}, the authors formulated the concept of minimally sized feedback destabilization attacks using the dynamical structure function introduced in \cite{GoncalvesDSF}. This theory was later applied to destabilize models of power grids \cite{TackettPower}, social networks \cite{Rogers2025}, and water distribution systems \cite{GrimsmanWater}. 

These prior works employ a frequency domain formulation, but there are advantages to working in state space. The dynamics of many critical systems can be naturally described by physical state relationships. In particular, vulnerabilities can be linked to physical quantities within the system, leading to actionable mitigation strategies. Additionally, working in state space allows us to give a careful description of the stability picture even when the system is nonlinear. 

Our main contributions are three-fold: 
\begin{itemize}
\item First, we review and assemble results from robust control theory into a general framework for the design of minimal destabilization attacks, with a focus on state space formulations, applicable to nonlinear infrastructure systems. 
\item Second, we demonstrate that, while minimal feedback attacks push the system to the stability boundary, they need not lead to instability of the underlying state space model when the system is nonlinear, {\em even though instability is guaranteed for linear systems}. 
\item Finally, we prove that when the gain of a minimal feedback attack is slightly increased, then internal destabilization is guaranteed, even in the nonlinear case. Although implicitly obvious, this result appears to require a nontrivial proof, which we provide. Our argument depends on a key regularity result for the unstable eigenvalue near the imaginary axis.

\end{itemize}

\section{PROBLEM FORMULATION}

Consider a nonlinear dynamic system given by:
\begin{equation}\label{orgnonlin}
    \begin{aligned}
        \dot x &= \bar{f}(x,u,w)\\
        y &= g_y(x,u,w) \\
        r &= g_r(x,u,w),
    \end{aligned}
\end{equation}
with $x(t) \in \mathbb{R}^n$ representing internal states, $u(t) \in \mathbb{R}^{k}$ describing nominal control inputs from system operators, $w(t) \in \mathbb{R}^m$ modeling an adversary's capability to input signals into the system, $y(t) \in \mathbb{R}^{l}$ representing 
measurements collected by system operators, and $r(t) \in \mathbb{R}^p$ being signals which are visible to the adversary.  This system might model the dynamics of a large-scale, complex infrastructure system, such as a network of gas pipelines, a power system, a water treatment facility, or interconnections of all the above. Unlike typical control-oriented models, this model not only considers paths of nominal use, through $u$ and $y$, but also paths of possible {\em misuse}, through $w$ and $r$.

We assume that system managers have already designed a suitable feedback controller mapping measurements $y(t)$ to controls $u(t)$, yielding the system
\begin{equation}\label{nonlin}
    \begin{aligned}
        \dot x &= f(x,w)\\
        r &= g(x,w).
    \end{aligned}
\end{equation}
We refer to \eqref{nonlin} as the \textit{exposed system}, which is the starting point for our analysis. It is useful to distinguish the special case of no adversarial input $w\equiv0$, in which case we refer to \eqref{nonlin} as the \textit{nominal system}. 

We next review relevant notions of stability.
\begin{defn}\label{stability}
    Consider the autonomous system 
    \begin{equation}
        \dot x = F(x).
    \end{equation}
    Suppose that $F(x_e) = 0$. Then the equilibrium $x_e$ is 
    \begin{itemize}
        \item Lyapunov stable if for every $\epsilon > 0$ there exists $\delta >0$ such that $|x(t) - x_e| < \epsilon$ for each $t \ge 0$ whenever $|x(0) - x_e| < \delta$.
        \item Asymptotically stable if it is Lyapunov stable and if there exists $\delta >0$ such that $x(t) \to x_e $ as $t \to \infty$ whenever $|x(0) - x_e| < \delta$.
        \item Unstable if it is not Lyapunov stable.
    \end{itemize}
\end{defn}
We assume that the nominal system is asymptotically stable at an equilibrium point $x_e$, which we may take to be the origin without loss of generality. Similarly, we may assume that the nominal system provides zero output at equilibrium: $g(0,0) = 0$. Finally, we will assume that $f(x,w)$ is sufficiently regular (say, $C^1$ in a neighborhood of $(x,w) = (0,0)$ ) and that
\begin{equation*}
    \dot x = f(x,w), \quad x(0) = x_0
\end{equation*}
admits a unique bounded solution for all time whenever $x_0$ and $w(t)$ are small enough.

After linearizing \eqref{nonlin} about $(x,w)=(0,0)$ we obtain %the near steady-state approximation
\begin{equation}\label{lti}
    \begin{aligned}
        \dot x &= Ax + Bw \\
        r &= Cx + Dw,
    \end{aligned}
\end{equation}
which is linear and time-invariant (LTI). We refer to this system as $G$, and use $G(s)$ to denote its  $p \times m$ transfer function
\begin{equation}
    G(s) = C(sI-A)^{-1}B + D, \quad s \in \mathbb{C} - {\text{spec}(A)}.
\end{equation}
When $x(0)=0$, then under suitable assumptions on $w(t)$, the function $G(s)$ maps the exposed inputs to the exposed outputs via multiplication in the Laplace domain: $R(s) = G(s)W(s)$. We will assume that $G(s)$ is not identically zero, which means that the attacker can observe an effect of at least one nonzero input. 

From Lyapunov's Indirect Method, we conclude that the eigenvalues of $A$ (and thus the poles of $G$) have non-positive real part (see e.g. \cite{CoddingtonODE}, \cite{HartmanODE}). 

\begin{thm}\label{stability_thm}
    {\rm\bf  [Lyapunov's Indirect Method]}  Consider $$\dot x = F(x),$$ where $F$ is continuously differentiable with an equilibrium point $F(x_e) = 0$. Let $A = D_x F(x_e)$ be the Jacobian of $F$ with respect to $x$. 
    \begin{itemize}
        \item If all eigenvalues of $A$ have negative real part, then the equilibrium $x_e$ is asymptotically stable. 
        \item If $A$ has at least one eigenvalue with positive real part then the equilibrium $x_e$ is unstable. 
    \end{itemize}
\end{thm}
\begin{rmk}
    We will assume that all the eigenvalues of $A$ in \eqref{lti} have negative real part, which is standard for a controlled system \eqref{nonlin}. For particular systems with eigenvalues on the imaginary axis, higher order analysis with LaSalle's invariance principle and Lyapunov direct methods can be used to assess stability. %We defer these considerations to future work.
\end{rmk}
The goal of the adversary is to interconnect their system, referred to as $\Delta$, with the exposed one in such a way that instability of the closed-loop is guaranteed. This can be compared with the classical feedback \textit{stabilization} problem (\cite{ZhouTextbook},  \cite{DullerudPaganiniTextbook}, \cite{DahlehTextbook}), where one starts with an unstable system and seeks to design a controller for which the closed-loop is stable, possibly subject to additional constraints. The analysis of destabilization proceeds similarly. While robust control theory allows one to consider more complicated attacks, it turns that LTI $\Delta$ are sufficient for destabilization and can be made minimal in a precise sense. Therefore, the system $\Delta$ will be given by
\begin{equation}\label{delta_realization}
    \begin{aligned}
        \dot{\tilde{x}} &= \tA \tilde{x} + \tB r \\
        w &= \tC \tilde{x} + \tD r,
    \end{aligned}
\end{equation}
with the $m \times p$ transfer function
\begin{equation}
    \Delta(s) = \tC (sI-\tA)^{-1} \tB + \tD.
\end{equation}
As a practical matter, LTI $\Delta$ are among the easiest attacks to implement, requiring only the four real-valued matrices in \eqref{delta_realization}. Of course, the adversary will require that their own system be internally stable. Thus, we will assume that $\tA$ is chosen with all eigenvalues having negative real part.

The feedback interconnection of $\Delta$ with $G$ is obtained by setting the input $r$ of $\Delta$ equal to the output $r$ of $G$ in \eqref{lti}. Similarly, the input $w$ of $G$ is set equal to the output $w$ of $\Delta$ in \eqref{delta_realization}. For simplicity, assume that $D=0$. The general case is treated very similarly, but one needs to add the well-posedness assumption
\begin{equation}\label{WP}
    \det(I - \tilde{D}D) = \det (I - \Delta(\infty) G(\infty)) \neq 0,
\end{equation}
for the linear interconnection. For the nonlinear closed loop to remain well-posed, additional assumptions are required. A sufficient condition is that $\sup_{x,u}|| \tilde{D} \frac{\partial g}{\partial u}(x,u)||_2 < 1$ which guarantees that $I - \tilde{D} \frac{\partial g}{\partial u}(x,u)$ remains invertible. Alternatively, $\tilde{D} = 0$ is sufficient, and we demonstrate how this can be achieved in Section \ref{sec_constr}. We will therefore assume that one of these conditions holds so that the nonlinear closed loop is robustly well-connected. 

In the present case, the dynamics of the interconnection are given by the linear ODE system
\begin{equation}\label{inter}
        \begin{aligned}
            \frac{\text{d}}{\text{d}t}
            \begin{bmatrix}
                x \\
                \tilde{x}
            \end{bmatrix}
            &= \begin{bmatrix}
                A + B \tD C & B \tC \\
                \tB C & \tA
            \end{bmatrix}
            \begin{bmatrix}
                x \\
                \tilde{x}
            \end{bmatrix}.
        \end{aligned}
\end{equation}

Similarly, the dynamics of the interconnection between $\Delta$ and the original, possibly nonlinear, exposed system \eqref{nonlin} are given by
\begin{equation}\label{nonlin_inter}
        \begin{aligned}
            \dot x &= f(x,  \tD Cx + \tC \tilde{x}) \\
            \dot{\tilde{x}} &=  \tB C x +  \tA \tilde{x}.
        \end{aligned}
\end{equation}

Since \eqref{inter} is the linearization of \eqref{nonlin_inter} about $(x,\tilde x) = (0,0)$, it follows from Theorem \ref{stability_thm} that the origin is an unstable equilibrium for \eqref{nonlin_inter} whenever the 2x2 block matrix in \eqref{inter} has an eigenvalue with positive real part. The LTI destabilization problem can therefore be reduced to choosing matrices $\tA, \tB, \tC, \tD$ to meet this objective. 

\begin{rmk}
    An eigenvalue of \eqref{inter} with non-negative real part is guaranteed to excite the nominal states, since $\Delta$ is internally stable. 
\end{rmk}

Using a block matrix determinant formula, the characteristic polynomial $P(s)$ of the dynamics in \eqref{inter} is easily described in terms of the transfer functions $G(s)$ and $\Delta(s)$ (see e.g. \cite{paper_char_poly}). Let  $p_A(s)$ and $p_{\tA}(s)$ denote the characteristic polynomials of $A$ and $\tA$. Then
\begin{equation}\label{char_poly}
    P(s) =  p_A(s)p_{\tilde{A}}(s)\det (I-\Delta(s)G(s)).
\end{equation}

Therefore, an eigenvalue $\lambda$ with $\Re(\lambda) \ge 0$ can be assigned by choosing $\Delta$ so that $$\det (I-\Delta(\lambda)G(\lambda)) = 0.$$ 

While many systems $\Delta(s)$ can accomplish this, it is also possible to ensure that the interaction of $\Delta$ with the exposed system is ``stealthy" or small in a well-defined sense. A natural measure of the adversary's footprint is the peak gain introduced by their system 
\begin{equation}
    ||\Delta || = \sup _{r \neq 0} \frac{||w||_{\text{out}}}{||r||_{\text{in}}}
\end{equation}
where $||\cdot ||_{\text{in}}$ and $||\cdot ||_{\text{out}}$ are norms on the input and output signals, respectively. When $||\cdot ||_{\text{in}} = ||\cdot ||_{\text{out}} = ||\cdot ||_{L^2}$, the induced system norm is known as the $H_\infty$ norm, written $||\Delta||_{H_\infty}.$ It can be conveniently expressed in terms of the transfer function:
$$ ||\Delta||_{H_\infty} =  \sup_{\Re(s) \ge 0} ||\Delta(s)||_2 = \sup_{w \in \mathbb{R} } ||\Delta(jw)||_2.$$ 

This norm is especially important, owing to the following version of the small gain theorem for unstructured perturbations (see e.g. Ch. 7 of \cite{DahlehTextbook}).
\begin{thm}\label{thm:small-gain}
    For all asymptotically stable LTI systems $\Delta$ satisfying $$||\Delta||_{H_\infty} < \frac{1}{||G||_{H_\infty}},$$ the feedback interconnection between $G$ and $\Delta$ is internally stable. 
\end{thm}

By internally stable, we mean that any realization of the interconnection (e.g. \eqref{inter} is asymptotically stable. Consequently, it is natural to quantify the adversary's impact by its $H_\infty$ norm. In particular, the stealthiest possible attacks would have size 
$$||\Delta||_{H_\infty} = \frac{1}{||G||_{H_\infty}}.$$

The catch is that such a $\Delta$ necessarily produces an eigenvalue on the imaginary axis and not at any point in $\{\Re(s) > 0\}$. To see this, consider the family of rational functions
$$P_\tau(s) = p_A(s)p_{\tilde{A}}(s)\det (I-\tau\Delta(s)G(s))$$
where $\tau\in \mathbb{R}$. All the roots of $P_0$ have negative real part. Suppose $P_1 = P$ has a root with positive real part. By continuity, there must be a $\tau_0$ with $0 < \tau_0 < 1$  such that $P_{\tau_0}$ and therefore $\det (I-\tau_0\Delta G)$ have a root on the imaginary axis. Consequently, $\tau_0 \Delta$ is destabilizing, contradicting minimality of $\Delta.$ These considerations motivate the following definition:

\begin{defn}\label{def:attack}
    A minimal destabilizing feedback attack on the system $G$ is a stable LTI system $\Delta$ satisfying
    \begin{itemize}
        \item[(i)] $\det (I - \Delta(jw_0) G(jw_0) )= 0$, for some $\omega_0 \in \mathbb{R}$
        \vspace{0.1in}
        \item[(ii)] $||\Delta||_{H_\infty} = \frac{1}{||G||_{H_\infty}}$
    \end{itemize}
\end{defn}
We refer to (i) as the destabilization condition and to (ii) as the minimality condition. Note carefully that, due to the purely imaginary eigenvalue, a minimal destabilizing feedback attack \textit{need not destabilize the original nonlinear system.} See Section \ref{sec_example} for an example.

\section{DESTABILIZATION ANALYSIS}
In view of Theorem \ref{stability_thm}, guaranteed destabilization of the  exposed system \eqref{nonlin} by feedback requires that at least one eigenvalue is moved strictly to the right of the imaginary axis. In what follows, we prove that this is always possible by slightly perturbing a  minimal feedback attack $\Delta$. The results follow from linear perturbation theory and will be stated in slightly more generality than is required.

We will use the following basic result of linear algebra. Let $||\cdot||$ be any induced matrix norm.
\begin{lemma}\label{lemma:matrix}
    If $M \in \mathbb{C}^{m \times m}$ has $||M|| = 1$ and $\det(I-M) = 0$, then the eigenvalue at $\lambda=1$ is semi-simple. 
\end{lemma}
\begin{proof}
    Let $v$ be an eigenvector for $\lambda=1$. If the eigenvalue is defective, then there exists a nonzero vector $w \in\mathbb{C}^m$ with $(M - I)w = v.$ Therefore $Mw = w + v$ and by induction $M^pw = w + p v$ for any positive integer $p.$ Taking norms of both sides gives $||w|| \ge || M^p w|| \ge p||v|| - ||w||$. Send $p \to \infty$ for a contradiction.
\end{proof}

We can apply Lemma \ref{lemma:matrix} with $M = \Delta(j\omega_0) G(j\omega_0)$, where $\Delta, \,\omega_0$ are as in Definition \ref{def:attack}. Since the eigenvalue $\lambda=1$ must be semi-simple, linear perturbation theory (see e.g. \cite{KatoPerturbation}), states that there are holomorphic functions $\lambda_1(s), \dots, \lambda_r(s)$ defined on a neighborhood of $s=j\omega_0$ such that each $\lambda_k(j\omega_0) = 1$ and such that each $\lambda_k(s)$ remains an eigenvalue of $M(s) := \Delta(s)G(s)$. In other words, if the eigenvalue at $s=j\omega_0$ is repeated $r$ times, then there exist $r$ holomorphic eigenvalue branches near this point. 

We pause to point out that the eigenvalues of a rational transfer function may fail to be differentiable in general, even at a peak frequency. For example, the transfer function
$$H(s) = 
\frac{1}{(s+1)^2}
\begin{bmatrix}
    0 & 1 \\ 
    s & 0
\end{bmatrix}$$
satisfies $||H||_{H_\infty} = ||H(0)||_2 = 1$. But the eigenvalues $$\lambda_{\pm}(s) = \pm\frac{\sqrt{s}}{(s+1)^2}$$ are not differentiable at $s=0$. Indeed, $I - H(0)$ is invertible, and Lemma \ref{lemma:matrix} is not applicable.

On the other hand, the matrix $M(s) = \Delta(s) G(s)$ can exhibit no such behavior, due to the semi-simple eigenvalue of $\lambda=1$. This regularity result allows us to prove the following theorem.

\begin{thm}\label{main}
Suppose that $M(s)$ is a square matrix of holomorphic functions on an open subset of the closed right half plane with $$\sup_{\text{\textup{Re}\,}s \,\ge 0} ||M(s)|| \le 1,$$ and that there exists $\omega_0 \in \mathbb{R}$ for which $I-M(j\omega_0)$ is singular. Either 
$$ \det ( I - M(s) ) = 0 \,\,\,\, \text{for all } \,\,\text{\textup{Re}}\, s \ge 0,$$
or there exists $\epsilon_0>0$ and a real analytic curve \newline $z: (-\epsilon_0, \epsilon_0) \to \mathbb{C}$ with 
    \begin{itemize}
        \item $z(0) = j\omega_0$.
        \item $\det [I-(1+\epsilon) M(z(\epsilon))] = 0$ for all $ \epsilon \in (-\epsilon_0, \epsilon_0)$.
        \item $\text{\textup{Re}} \,(z(\epsilon)) > 0$ for all $ \epsilon \in (0, \epsilon_0)$.
    \end{itemize}

\end{thm}
\begin{proof}
    Suppose that $\det(I - M(s))$ is not zero everywhere on the closed right half plane. We may assume that $\omega_0=0$, else we consider $N(s) = M(s+j\omega_0).$ Let $\lambda$ be any one of the holomorphic eigenvalue branches, defined on a open ball $B \subset \mathbb{C}$ centered at $s=0$.
    
    The strategy is to apply the (analytic) implicit function theorem to the function $f : \mathbb{R} \times B \to \mathbb{C}$ given by $$f(\epsilon,s) = (1+\epsilon)\lambda(s) - 1.$$ 
    Note that $f(0,0) = 0$ and that $f(\epsilon,s) = 0$ implies that $I - (1+\epsilon) M(s)$ is singular. 
    
    To apply the implicit function theorem, we need to prove that $f_s(0,0) \neq 0$ or equivalently that $\lambda'(0) \neq 0$. To see this, observe that when $s \in B\cap \{ \Re(s) \ge 0\}$, we have $$ |\lambda(s)| \le \sup_{\Re(s) \ge 0} ||M(s)|| \le 1 $$ 
    by assumption. Therefore, $$ |\lambda(0) | = \max_{B\cap \{\Re(s)\ge 0\}}|\lambda(s)| = 1$$ is obtained on the boundary. By the anti-calculus proposition of complex analysis, we have that $\lambda'(0) \neq 0$ unless $\lambda(s)$ is constant on an open subset $U$ of this half-disk. The latter would mean that $\det(I-M(s)) \equiv 0$ on $U$, which forces $\det(I-M(s)) \equiv 0$ on the closed right half plane. This is ruled out by assumption.

    It follows that the implicit function theorem furnishes an $\epsilon_1 > 0$ and a (unique) curve $z : (-\epsilon_1, \epsilon_1) \to \mathbb{C}$ satisfying $f(\epsilon, z(\epsilon)) = 0$ for all $|\epsilon| < \epsilon_1$. Additionally, $z$ solves the differential equation
    \begin{equation}\label{zdot}
    \dot z(\epsilon) = - \frac{f_\epsilon(\epsilon, z(\epsilon))}{f_s(\epsilon, z(\epsilon))} = -\frac{1}{1+\epsilon} \frac{1}{\lambda'(s)}, \quad z(0) = 0.
    \end{equation}

    We claim that $\Re(z(\epsilon)) > 0$ when $\epsilon > 0$ is small enough. This can shown with first derivative techniques. From \eqref{zdot} we have $\dot z (0)  = -\frac{1}{\lambda'(0)}.$ To analyze this quantity, decompose $\lambda(s) = u (s) + j v(s) = u(x,y) + j v(x,y)$ into real and imaginary parts. Observe that $u(0,0) = 1$ and $v(0,0) = 0$. By the Cauchy-Riemann equations,
    \begin{equation*}
        \begin{aligned}
            \lambda'(0) &= u_x(0,0) + j v_x(0,0) \\
            &= u_x(0,0) - j u_y(0,0) \\
            &= u_x(0,0).
        \end{aligned}
    \end{equation*}
    
    The final equation follows from the fact that the real-valued function $u(0,y) = \Re(\lambda(jy))$ has a local maximizer at $y = 0$. For all small $x \ge 0$ the maximum principle gives
    \begin{equation*}
        \begin{aligned}
            u^2(x,0) + v^2(x,0) &= 
            |\lambda(x)|^2 \\
            &\le \sup_{B  \,\cap \, \Re(s) \ge 0} |\lambda(s)|^2 \\
            &= |\lambda(0)|^2 \\
            &= u^2(0,0) + v^2(0,0).
        \end{aligned}
    \end{equation*}
    This means that
    \begin{equation*}
        \begin{aligned}
            0 &\ge \partial_x [u^2(x,0) + v^2(x,0)]\\ 
            &= 2u(x,0) \cdot u_x(x,0) + 2v(x,0) \cdot  v_x(x,0) \\ 
            &= 2 u_x(0,0)
        \end{aligned}
    \end{equation*}
    at $x=0$. Thus $\lambda'(0) = u_x(0,0) \le 0$. Since we already showed that it cannot be zero, the only option is that $\lambda'(0) < 0$. From \eqref{zdot} we conclude that $\dot{z}(0) > 0$. Hence $\Re(z(\epsilon))$ is strictly increasing for sufficiently small $\epsilon > 0$, with $\Re(z(0)) = 0$. The result follows.
\end{proof}

\begin{cor}\label{cor:destab}
     Given the asymptotically stable LTI system $G$ and an $H_\infty$ minimal destabilizing feedback attack $\Delta$ on $G$, the feedback interconnection between $G$ and $\Delta_\epsilon := (1+\epsilon)\Delta$ is well posed and internally unstable for all sufficiently small $\epsilon>0$. Consequently, the feedback interconnection of the exposed system \eqref{nonlin} with $\Delta_\epsilon$ is well-posed and has the  origin as an unstable equilibrium point for all sufficiently small $\epsilon>0$. 
\end{cor}

\begin{proof}
This follows from applying Theorem \ref{main} to $M(s) = \Delta(s) G(s)$. Note that the case $\det(I - M(s)) \equiv 0$ is ruled out by our well-posedness assumption \eqref{WP} on the interconnection between $G$ and $\Delta$. Moreover, well-posedness continues to hold for the closed loops by continuity.
\end{proof}

\section{UNSTRUCTURED FEEDBACK ATTACKS}\label{sec_constr}
In this section, we give a detailed construction of unstructured feedback attacks. This construction below is based on the proof of necessity for the small gain theorem in Ch. 9 of \cite{ZhouTextbook}.

\subsection{Single Input, Single Output Case}
When $m = k = 1$, then $G(s)$ is simply a rational function, and $H_\infty$ minimal feedback attacks are relatively simple to construct. The following simple result will be useful.
\begin{lemma}\label{DeltaConstr}
    Let $z$ be a complex number, and let $\omega_0$ be real. There there exists a stable, proper rational function $r(s)$ with real coefficients satisfying
    \begin{itemize}
        \item $r(j\omega_0) = z$
        \item $|r(j\omega)| = |z|$ for all $\omega \in \mathbb{R}$.
    \end{itemize}
    In particular, 
    \begin{itemize}
        \item If $z \in \mathbb{R}$ then we may take $r \equiv z.$
        \item If $z \notin \mathbb{R}$ then we may take 
            \begin{equation}\label{def_alpha_sigma}
                r(s) = \sigma|z|\cdot \frac{s - \alpha}{s + \alpha}, 
            \end{equation}
            where $\sigma = \text{\textup{sgn}} (\omega_0 \, \text{\textup{Im} \,z})$ and $\alpha = \frac{|\omega_0 \text{\textup{Im}}\,z|}{|z| + \sigma \text{\textup{Re}}\,z} > 0.$
    \end{itemize}
\end{lemma}
By stable and proper, we mean that any poles have negative real part and that the denominator has degree at least as large as the numerator. To construct $\Delta$, simply observe that since $G(s)$ is strictly proper, we have
\begin{equation}
    ||G||_{H_\infty} = \sup_{\omega \in \mathbb{R}} |G(j\omega)| = |G(j\omega_0)|,
\end{equation}
for some critical frequency $\omega_0 \in \mathbb{R}$. Now apply Lemma \ref{DeltaConstr} with $z = \frac{1}{G(j\omega_0)}$ to obtain $\Delta(s) := r(s)$. It is clear that $1 - \Delta(j\omega_0) G(j\omega_0) = 0$, and $||\Delta||_{H_\infty} = ||G||_{H_\infty}^{-1}$. Therefore, the destabilization condition and minimality are fulfilled.

In the case that $\Delta(s)$ is the real constant $1/G(j\omega_0)$, we have a stateless realization
\begin{equation} \label{statelessrealize}
    (\tA, \tB, \tC, \tD) = \left([\,], [\,], [\,], \frac{1}{G(j\omega_0)}\right).
\end{equation}

In the case that $G(j\omega_0) \not\in \mathbb{R}$,
\begin{equation*}
    \begin{aligned}
        \Delta(s) &= \frac{\sigma}{|G(j\omega_0)|} \frac{s- \alpha}{s + \alpha} =  \frac{\sigma}{|G(j\omega_0)|} \left(1 - \frac{2\alpha}{s+ \alpha} \right),
    \end{aligned}
\end{equation*}
and we can take, for instance,
\begin{equation}
    (\tA, \tB, \tC, \tD) = \left(-\alpha, 2\alpha, \frac{-\sigma}{|G(j\omega_0)|}, \frac{\sigma}{|G(j\omega_0)|}\right).
\end{equation}

\subsection{Multi Input, Multi Output Case}
For a strictly proper $p \times m$ transfer function $G(s)$, we proceed as follows. As before, let $\omega_0 \in \mathbb{R}$ be a critical frequency. That is,
\begin{equation}
    ||G||_{H_\infty} = \sup_{\omega \in \mathbb{R}} \sigma_{\text{max}}(G(j\omega)) = \sigma_{\text{max}}(G(j\omega_0)).
\end{equation}

Consider the singular value decomposition
\begin{equation}
    G(j\omega_0) = U \Sigma V^*,
\end{equation}
where $U$ is size $p \times p$, $V$ is size $m \times m$, and $\Sigma$ is size $p \times m$ with zeroes away from the main diagonal $\sigma_1, \dots, \sigma_{\text{min}(p,m)}$. Take
\begin{equation}\label{crit_delta}
    \Delta(j\omega_0) = V \tilde\Sigma U^*,
\end{equation}
where $\tilde\Sigma$ is size $m \times p$ with $\tilde\Sigma_{1,1} = \sigma_1^{-1}$, and all other entries are zero. Then
\begin{equation*}
    \begin{aligned}
        I - \Delta(j\omega_0) G(j\omega_0) &= I - V \tilde\Sigma U^*U \Sigma V^* \\
        &= I - V \tilde\Sigma \Sigma V^* \\
        &= V\left(I -
        \begin{bmatrix}
            1 & 0 & \cdots & 0 \\
            0 & 0 & \cdots & 0 \\
            \vdots & \vdots & \ddots & \vdots \\
            0 & 0 & \cdots & 0
        \end{bmatrix}\right)V^*, 
    \end{aligned}
\end{equation*}
which has $\lambda = 0$ as a simple eigenvalue $\lambda=0$ and the only other eigenvalue $\lambda=1$ repeated $m-1$ times. In particular, $I - \Delta(j\omega) G(j\omega_0)$ is singular.

To define $\Delta(s)$ in general, first note that \eqref{crit_delta} can be re-written using the left and right singular vectors $u$ and $v$ of $G(j\omega_0)$ corresponding to $\sigma_1$:
\begin{equation}
    \Delta(j\omega_0) = \frac{1}{\sigma_1} v u^*.
\end{equation}

Write $u = [u_1, \dots, u_p]^T$ and $v = [v_1, \dots, v_m]^T$. Now apply Lemma \ref{DeltaConstr} to get stable, proper rational functions $a_1(s), \dots, a_m(s)$ and $b_1(s), \dots, b_p(s)$ with real coefficients satisfying 
\begin{equation*}
    \begin{aligned}
        a_i(j\omega_0) &= v_i \\
        b_i(j\omega_0) &= u_i^* \\
        |a_i(j\omega)|, |b_i(j\omega)| &\text{ are constant}.
    \end{aligned}
\end{equation*}
Then we define
\begin{equation}
    \Delta(s) = \frac{1}{\sigma_1} a(s) b(s)^T,
\end{equation}
where $a(s)=[a_1(s), \dots, a_m(s)]^T$ and $b(s)=[b_1(s), \dots, b_p(s)]^T$.

We have already shown that $I-\Delta(j\omega_0)G(j\omega_0)$ is singular. Minimality follows from $$||\Delta(s)||_2 \le  \sigma_1^{-1} = ||\Delta(j\omega_0)||_2. $$

\subsection{The Case $D \neq$ 0}
The previous constructions depends on a critical frequency $\omega_0$. A finite critical frequency is guaranteed to exist when $D=0$, but for nonzero $D$, an infinite critical frequency is possible. In the case where there is both a finite and infinite critical frequency, one can use the above construction at a finite critical frequency $\omega_0$. Well-posedness is then maintained by multiplying $\Delta$ with a strictly proper scalar filter $r(s)$ with $r(j\omega_0) = 1$ and $|r(j\omega)| \le 1$. This is accomplished, for instance, by 
\begin{equation*}
    r(s) = \frac{2|\omega_0|s}{(s+ |\omega_0|)^2} \quad\text{ and }\quad r(s) = \frac{1}{s+1}  
\end{equation*}
for the cases $\omega_0 \neq 0$ and $\omega_0 = 0$, respectively.

When there is only a critical frequency at infinity, one can instead construct a feedback attack $\Delta$ which is arbitrarily close to minimal by applying the previous procedures with $\omega_0$ and $\epsilon>0$ satisfying $(1+\epsilon)|| G(j\omega_0)||_2 \ge ||G||_{H_\infty}$. This interconnection is well-posed by construction.

\section{ILLUSTRATIVE EXAMPLE}\label{sec_example}
Consider the nonlinear system 
\begin{equation}\label{ex:nonlin}
    \begin{aligned}
        \dot x_1 &= x_2\\ 
        \dot x_2 &= -x_1 - x_2 - x_2^3 + w\\
        r &= x_2
    \end{aligned}
\end{equation}
The origin is asymptotically stable for the nominal system, owing to the Lyapunov function 
$$V(x) = \frac12 \left(x_1^2+x_2^2\right),$$
where
\begin{equation*}
    \begin{aligned}
        \dot V(x) &= x_1 \dot x_1 + x_2 \dot x_2 \\
        &= x_1 x_2 + x_2\left(-x_1 - x_2 - x_2^3\right) \\
        &=  x_1 x_2  -x_1x_2 - x_2^2 - x_2^4 \\
        & = -x_2^2 - x_2^4 \le 0
    \end{aligned}
\end{equation*}
and $\dot V(x) = 0$ only when $x=0$. The linearization is given by
\begin{equation}
    \begin{aligned}
        \frac{\text{d}}{\text{d}t} 
        \begin{bmatrix}
            x_1 \\
            x_2
        \end{bmatrix} &= 
        \begin{bmatrix}
            0 & 1 \\
            -1 & -1
        \end{bmatrix}
        \begin{bmatrix}
            x_1 \\
            x_2
        \end{bmatrix} + 
        \begin{bmatrix}
            0 \\
            1
        \end{bmatrix} w\\
    r &= 
    \begin{bmatrix}
        0 & 1
    \end{bmatrix}
    \begin{bmatrix}
            x_1 \\
            x_2
    \end{bmatrix}.
    \end{aligned}
\end{equation}
The transfer function 
\begin{equation}
    G(s) = \frac{s}{s^2 + s + 1}
\end{equation}
has critical frequencies $\omega_0 = \pm 1$. Since $G(\pm 1j) = 1$ is real, a minimal destabilizing feedback attack is given by \mbox{$\Delta = \frac{1}{G(1j)} = 1 $} (see \eqref{statelessrealize}). The linear closed loop has the system matrix
\begin{equation*}
    A = 
    \begin{bmatrix}
        0 & 1 \\
        -1 & 0
    \end{bmatrix}
\end{equation*}
with eigenvalues $\lambda = \pm i$.  That is to say, this $\Delta$ in feedback with the linearized system is unstable, as expected. However, the closed loop with $\Delta$ and the nonlinear system \eqref{ex:nonlin} has the dynamics
\begin{equation*}
    \begin{aligned}
        \dot x_1 &= x_2\\ 
        \dot x_2 &=   -x_1 - x_2 - x_2^3 + x_2 \\
        &= -x_1 -x_2^3.
    \end{aligned}
\end{equation*}
This system remains \textit{asymptotically stable}, as the same Lyapunov function attests
\begin{equation*}
    \begin{aligned}
        \dot V(x) &= x_1 x_2 + x_2\left(-x_1  - x_2^3\right)\\
        &=- x_2^4.
    \end{aligned}
\end{equation*}
However, the perturbation $\Delta_\epsilon = (1+\epsilon)\Delta = (1+\epsilon)$ will destabilize the closed loop with \eqref{ex:nonlin}, as Corollary \ref{cor:destab} guarantees. Indeed, the linear closed loop has the matrix
\begin{equation*}
    A = 
    \begin{bmatrix}
        0 & 1 \\
        -1 & \epsilon
    \end{bmatrix},
\end{equation*}
with eigenvalues $\lambda_\pm = \frac12 \left(\epsilon \pm \sqrt{\epsilon^2 - 4}\right)$. Since $\Re \, \lambda_\pm > 0$, Theorem \ref{stability_thm} also guarantees the instability. The same Lyapnuov function fails due to
\begin{equation*}
    \dot V(x) = x_2^2(\epsilon-x_2^2),
\end{equation*}
which is positive for $ 0 < |x_2|< \epsilon^\frac12$.

This demonstrates how a minimal destabilizing $\Delta$ for the {\em linearization} of a nonlinear system may not actually destabilize the nonlinear system itself, while nudging the $\Delta$ by a gain of $1+\epsilon$ is guaranteed to do so. 

\section{CONCLUSION}

In this paper, we assembled existing results from robust control literature into a general framework for the analysis of minimal destabilizing controllers. Additionally, we explored the disparity in instability guarantees between a nonlinear system and its linearization under the influence of a such a controller. We demonstrated that minimal ``destabilizers" can still result in an asymptotically stable closed loop with the nonlinear system. However, we proved that these malicious controllers can be perturbed slightly in order to guarantee instability of the nonlinear closed loop.  Note that this theory can be extended to the case of structured attacks using $\mu$-analysis.

These results may be used to assess the robustness of nonlinear systems to feedback attacks and to identify critical vulnerabilities, leading to more resilient systems. Future work in this direction could explore attack design from incomplete system information, detectability and robustness of destabilizers, and adversarial resource constraints. 

\section{Acknowledgments}
The authors thank Sean Warnick for his mentorship and helpful comments on the manuscript.

\bibliographystyle{IEEEtran}
\bibliography{IEEEexample}
\end{document}